\documentclass[a4paper,leqno,11pt,twoside]{amsart}
\usepackage{amsmath,amsfonts,amssymb,amsthm,amscd}
\usepackage[utf8]{inputenc}
\usepackage[english]{babel}
\usepackage[colorlinks=true,citecolor=blue, urlcolor=blue, linkcolor=blue,pagebackref]{hyperref}
\usepackage[top=1.1in,bottom=1.1in,left=1.in,right=1.in]{geometry} 
\usepackage{graphicx}
\usepackage{wrapfig}
\usepackage{paralist}
\usepackage{tabto}
\usepackage{standalone}
\usepackage{xfrac}
\usepackage{float}
\usepackage{tikz-cd}
\usepackage{mathtools}

\usepackage{pgf,tikz}
\usetikzlibrary{arrows,chains,
 decorations.pathreplacing,
shapes,%
}

\usepackage[all]{xy} 

\usepackage{enumerate}
\usepackage{xcolor}
\usepackage{aliascnt}

\theoremstyle{plain}

\newtheorem{thmIntr}{Theorem}

\newaliascnt{propIntr}{thmIntr}

\newaliascnt{corIntr}{thmIntr}

\newaliascnt{QU}{thm}

\aliascntresetthe{QU}

\newaliascnt{lem}{thm}
\newtheorem{lem}[lem]{Lemma}
\aliascntresetthe{lem}

\newaliascnt{cor}{thm}
\newtheorem{cor}[cor]{Corollary}
\aliascntresetthe{cor}

\newaliascnt{prop}{thm}
\newtheorem{prop}[prop]{Proposition}
\aliascntresetthe{prop}

\theoremstyle{definition}

\newaliascnt{rem}{thm}
\newtheorem{rem}[rem]{Remark}
\aliascntresetthe{rem}

\newaliascnt{defn}{thm}

\aliascntresetthe{defn}

\newaliascnt{ex}{thm}

\aliascntresetthe{ex}

\numberwithin{equation}{section}

\def\bP{\ensuremath{\mathbb{P}}}

\def\bZ{\ensuremath{\mathbb{Z}}}

\def\cI{\ensuremath{\mathcal{I}}}

\def\cO{\ensuremath{\mathcal{O}}}
\def\cP{\ensuremath{\mathcal{P}}}

\def\tC{\ensuremath{\widetilde{C}}}

\def\Db{\mathop{\mathrm{D}^{\mathrm{b}}}\nolimits}

\DeclareMathOperator{\Pic}{Pic}

\DeclareMathOperator{\sExt}{\mathcal{E}\emph{xt}}
\DeclareMathOperator{\Hom}{Hom}
\DeclareMathOperator{\sHom}{\mathcal{H}\emph{om}}
\DeclareMathOperator{\NS}{NS}
\DeclareMathOperator{\rk}{rk}
\DeclareMathOperator{\Id}{Id}

\DeclareMathOperator{\ch}{ch}

\DeclareMathOperator{\irr}{irr}
\DeclareMathOperator{\Gr}{Gr}

\DeclareMathOperator{\Nm}{Nm}

\def\longarrow#1#2{\mathchoice{#2}{#1}{#1}{#1}}
\def\to{\longarrow{\rightarrow}{\longrightarrow}}

\let\shortmapsto\mapsto
\def\mapsto{\longarrow{\shortmapsto}{\longmapsto}}

\definecolor{applegreen}{rgb}{0.55, 0.71, 0.0}

\title[Picard bundles and the degree of irrationality of Jacobians and Pryms]{Picard bundles and the degree of irrationality of Jacobians and Pryms}
\author[F.~Moretti and A.~Rojas]{Federico Moretti and Andr\'es Rojas}

\address{Federico Moretti: Department of Mathematics, State University of New York at Stony Brook \hfill\newline\texttt{}
\indent 100 Nicolls Rd, Stony Brook, NY 11794, USA}
\email{federico.moretti@stonybrook.edu}

\address{Andrés Rojas: Departament de Matemàtiques i Informàtica, Universitat de Barcelona \hfill\newline\texttt{}
\indent Gran Via de les Corts Catalanes 585, 08007 Barcelona, Spain}
\email{andresrojas@ub.edu}

\begin{document}

\begin{abstract}
For a smooth projective curve of genus $g$, we study some positivity properties of (twisted)  Picard bundles on the symmetric product. As applications, we prove that the degree of irrationality of any genus $g$ Jacobian is bounded from above by $2^g$, and the degree of irrationality of any $(g-1)$-dimensional Prym variety is bounded from above by $2^{2g-3}$.
\end{abstract}

\maketitle

\setcounter{tocdepth}{1}

\section{Introduction}

Let $C$ be a smooth projective curve of genus $g\ge 2$. On its Jacobian $JC$ one can construct several vector bundles, usually referred to as \emph{Picard bundles}, whose fibers parametrize the cohomology groups of all line bundles of a fixed degree on $C$. These vector bundles were introduced in classical works of Mattuck \cite{mattuck} and Schwarzenberger \cite{schw}, and redefined by Mukai \cite{mukai-duality} in terms of the Fourier--Mukai transform. Beyond their natural appearance in the determinantal realization of the Brill--Noether loci $W^r_d(C)$, they have been explored from several perspectives, including stability \cite{kempf,ein-laz}, cohomological regularity and continuous global generation \cite{parpopa-reg,pareschi-generation}, or potential approaches to the Schottky problem \cite{debarre}.

In this note we focus on the rank-$g$ Picard bundle $F$ whose fiber over $\alpha\in JC$ admits a canonical identification
\[
F|_\alpha \cong H^1\!\big(C,\cO_C(-p_0)\otimes \alpha\big),
\]
where $p_0\in C$ denotes a fixed base point\footnote{For the purposes of this work, the line bundle $\mathcal{O}_C(-p_0)$ can be replaced by any line bundle of degree $-1$.}. We investigate some elementary properties of $F$ by passing to the symmetric product $C^{(g)}$.
More precisely, consider the Abel--Jacobi map
\[
\pi:C^{(g)}\to \Pic^0(C)=JC,\qquad p_1+...+p_g\longmapsto \cO_C(p_1+...+p_g-gp_0)
\]
and denote by
\[
x:=p_0+C^{(g-1)}\subset C^{(g)}
\]
the standard divisor defined by $p_0$. 
We prove that the vector bundle $\widetilde{F}:=\pi^*F\otimes\cO_{C^{(g)}}(x)$ on $C^{(g)}$ is globally generated and satisfies $c_g(\widetilde{F})=2^g$, and we give a precise description of the zero locus of global sections of $\widetilde{F}$ from which we constraint the rank of its determinant map.




The main geometric application of our analysis of $\widetilde{F}$ is the following upper bound on the degree of irrationality of Jacobian varieties:

\begin{thmIntr}\label{thm:intro-jacobian-bound}
For any smooth curve $C$ of genus $g\ge 2$, the inequality $\irr(JC)\le 2^g$ holds.
\end{thmIntr}

The \emph{degree of irrationality} $\irr(X)$ of an integral projective variety $X$  is the minimal degree of a rational dominant map
$X\dashrightarrow \mathbb P^{\dim(X)}$. This birational invariant, generalizing to higher dimension the classical notion of gonality of a curve, has recently been the object of considerable activity (see \cite{CM} for a survey on recent work).
However, an essentially complete solution is known in only very few cases, such as hypersurfaces in projective space \cite{BDELU,fanohyp} or abelian surfaces \cite{nathchen,martin}. 
For abelian varieties of dimension $\geq3$ one has nontrivial lower bounds \cite{av}, but relevant upper bounds are not known; even for Jacobians, the standard upper bound comes from symmetrizing a map $C\to\bP^1$ and gives $\irr(JC)\le \mathrm{gon}(C)^g\le \lfloor \frac{g+3}{2}\rfloor^g$, see \cite[Problem 4.6]{CM}.

The key point in the proof of \autoref{thm:intro-jacobian-bound} is the fact that a globally generated vector bundle $E$ of rank $\dim(X)$ with positive top Chern class produces, via a determinantal construction,
a rational dominant map to $\bP^{\dim(X)\\}$ whose degree is bounded by the top Chern class of $E$ (see \autoref{kernel-bdles}, building on results from \cite{mor}). Applying this general principle to the vector bundle $\widetilde{F
}$ on the symmetric product $C^{(g)}$ results in the upper bound $\irr(JC)\leq 2^g$.

Notably, the rational maps $JC\dasharrow \bP^g$ we construct are not rigid. For a fixed divisor $\Theta$ representing the canonical principal polarization on $JC$, we obtain a large family---parametrized by an open subset of the Grassmannian $\Gr(g+1,2g)$---of linear projections
\[
JC\hookrightarrow |(g+1)\Theta|^\vee\dasharrow \bP^g
\]
of degree $2^g$. It turns out that such projections admit a concrete geometric realization: considering the closed immersion
\[
C\hookrightarrow \bP^{2g-1}=|\omega_C((g+1)p_0)|^\vee
\]
and fixing a linear subspace $\bP^g\subset \bP^{2g-1}$, the corresponding map $C^{(g)}\dasharrow \bP^g$ sends a general $p_1+...+p_g\in C^{(g)}$ to the intersection point of $\bP^g$ with the linear span $\langle p_1,...,p_g\rangle\subset \bP^{2g-1}$. We will address this geometric perspective for arbitrary symmetric products of curves and Brill--Noether loci in future work.

In the last part of the paper we exploit a similar viewpoint, namely the use of suitable Picard type vector bundles, to find an upper bound for the degree of irrationality of Prym varieties. Recall that, for a smooth curve $C$ of genus $g$ and an irreducible \'etale double cover $f:\widetilde C\to C$, the associated Prym variety $\Pr(f)$ as the connected component of
$\ker(\Nm_f:J\widetilde C\to JC)$
containing the origin of $J\tC$.
It is a principally polarized abelian variety of dimension $g-1$. We obtain the following:

\begin{thmIntr}\label{thm:intro-prym-bound}
For any $(g-1)$-dimensional Prym variety $\Pr$ the inequality $\irr(\Pr)\leq 2^{2g-3}$ holds.
\end{thmIntr}

In view of our results, and to some extent of our techniques, it would be tempting to speculate upper bounds for the degree of irrationality of a $d$-dimensional principally polarized abelian variety $(A,\Theta)$ in terms of $d$ and of the minimal integer $k$ such that the class $k\frac{\Theta^{d-1}}{(d-1)!}$ is represented by an effective curve. Of course, a first step in this direction should be an analysis of Picard type bundles on Prym--Tjurin varieties.

\textbf{Acknowledgements.} We thank Gabi Farkas and Rob Lazarsfeld for some discussions related to this circle of ideas. We are especially grateful to Martí Lahoz and Joan Carles Naranjo for suggesting the strategy to attack \autoref{thm:intro-prym-bound}. Rojas was supported by the Spanish MICINN project PID2023-147642NB-I00.

\section{Preliminaries}

\subsection{Rational maps via kernel bundles.}\label{kernel-bdles}
Our starting point is the observation that, for a given smooth projective variety $X$ of dimension $n$, one can study rational maps $X\dasharrow \bP^n$ via the associated kernel sheaf. Namely, given $L\in\Pic(X)$ and a linear system $V\in \mathrm{Gr}(n+1,H^0(L))$ without base divisors, one can consider the reflexive sheaf $K$ sitting in an exact sequence
\[
\begin{tikzcd}
    0 \arrow{r} & K \arrow{r} & V \otimes \mathcal O_X \arrow{r}{ev} & L \otimes \mathcal I \arrow{r} & 0,
\end{tikzcd}
\]
where $B:=\underline{\mathop{\mathrm{Spec}}}(\cO_X/\cI)$ denotes the base locus of the rational map $\varphi_V:X \dashrightarrow \bP^n=\mathbb P(V^\vee)$. Dualizing yields an inclusion $V^\vee \subset H^0(K^\vee)$; note that $V^\vee$ generates $K^\vee$ in codimension $1$ (away from $B$) and also that $H^0(K)=0$.

More generally, following \cite{mor} 
one can consider \emph{good pairs} $(E,V^\vee)$ where $E$ is a rank $n$ reflexive sheaf with $H^0(E^\vee)=0$, and $V^\vee\in \mathrm{Gr}(n+1,H^0(E))/\mathrm{Aut}(E)$ generates $E$ in codimension $1$. The following holds:

\begin{prop}[{\cite[Proposition 1.5]{mor}}]\label{compdegree}

	Let $X$ be a smooth projective variety of dimension $n$. There is a \emph{one-to-one} correspondence\[
        \big \{ \textrm{rational non-degenerate maps } X\dashrightarrow \mathbb P^n\big \} \longleftrightarrow \big \{ \textrm{good pairs } (E,V^\vee)  \big \}
        \]
        \[
       \quad \quad \quad  \quad  \quad  \quad \quad \, \,  \, \, \, \, (L,V)\longmapsto (K^\vee,V^\vee).
        \]
       Moreover, the fibers of the rational map $\varphi_V:X\dashrightarrow \mathbb P(V^\vee)$ are described by the rule
       \[
       \varphi_V^{-1}([s])\cap (X\setminus B)=Z(s)\cap (X\setminus B),
       \]
       where $s$ can be seen both  as a point of the target projective space (left-hand side of the equality) and as a global section of the reflexive sheaf $E$ (right-hand side of the equality).
\end{prop}

For varieties of dimension $\ge 3$, computing the degree of a rational map (even understanding if a map is generically finite) using the description above can be difficult. However, the situation gets simpler when dealing with globally generated vector bundles:

	\begin{lem}
		Let $X$ be a smooth projective variety of dimension $n$, and let $E$	be a globally generated vector bundle of rank $n$ with $c_n(E)>0$. Then for a general $V^\vee\in \mathrm{Gr}(n+1,H^0(E))$ the induced rational map $\varphi_V:X \dashrightarrow  \mathbb {P}(V^\vee)$ is generically finite of degree $c_n(E)$.	\end{lem}
	\begin{proof}
		First note that $H^0(E^\vee)=0$; indeed, if $H^0(E^\vee)\neq0$ then by global generation $\cO_X$ must be a direct summand of $E$, which implies $c_n(E)=0$. Furthermore, $h^0(E)\geq n+1$ (as $h^0(E)\leq n$ would imply $E\cong \cO_X^{\oplus n}$, again by global generation of $E$).

        It is well known that a general $V^\vee\in \Gr(n+1,H^0(E))$ generates $E$ in codimension 1. To see this, note that for a general $W\in \Gr(n,H^0(E))$ the locus $\Gamma_W$ where $W$ does not generate $E$ is generically reduced of codimension 1; fix one such $W$, and pick a general point $p\in X$ in some irreducible component of $\Gamma_W$. For a general section $s\in H^0(E)$, the subspace $W+\langle s\rangle$ generates $E$ at $p$; therefore, after doing this for finitely many points (one in each irreducible component of $\Gamma_W$), the claim follows.

        Therefore, $(E,V^\vee)$ is a good pair for a general $V^\vee\in\Gr(n+1,H^0(E))$. Moreover, since a general global section of $E$ has 0-dimensional zero locus (of length $c_n(E)$), according to \autoref{compdegree} the induced rational map $\varphi_V:X \dashrightarrow  \mathbb {P}(V^\vee)$ is dominant of degree $\leq c_n(E)$. To prove $\deg(\varphi_V)=c_n(E)$, we only need to check that for a general $V^\vee\in\Gr(n+1,H^0(E))$ and $s\in V^\vee$,  $Z(s)$ is disjoint from the (codimension $\geq2$) locus $B$ where $V^\vee$ does not generate $E$.

        Consider a general $s\in H^0(E)$, such that $Z(s)$ is a finite subscheme supported at the points $p_1,\dots, p_k\in X$. Since $E$ is globally generated, for any $i=1,\dots,k$ we can find a (non-empty) open subset $U_i\subset \mathrm{Gr}(n,H^0(E))$ such that  $W\otimes \mathcal O_X\to E$ is surjective at $p_i$ for every $W\in U_i$. Therefore taking $W\in \bigcap_{i=1}^k U_i$, we can consider $V^\vee=\langle s \rangle \oplus W$; since $s$ and $W$ are chosen general, $V^\vee$ is general in $\mathrm{Gr}(n+1,H^0(E))$. In particular $V^\vee$ generates $E$ in codimension $1$, and by construction $Z(s)$ avoids the locus $B$ where $V^\vee$ does not generate $E$. In view of \autoref{compdegree} the fiber above the class of $[s]$ is  $Z(s)\cap (X \setminus B)=Z(s)$  of degree $c_n(E)$. Since the property of vanishing along a finite subscheme which avoids $B$ is open in $\bP(V^\vee)$, the proof is complete.
 		\end{proof}

\begin{cor}\label{deg-gg}
Let $X$ be a smooth projective variety of dimension $n$, carrying a globally generated vector bundle $E$ of rank $n$ with $c_n(E)>0$. Then, $\mathrm{irr}(X)\le c_n(E)$.
\end{cor}

\subsection{Picard bundles on Jacobians} Let $(A,\theta)$ be a principally polarized abelian variety, namely $\theta=c_1(L)\in\NS(A)$ for an ample line bundle $L$ with $h^0(A,L)=1$. We have the isomorphism
\[
\varphi_\theta:A\overset{\cong}{\longrightarrow} A^\vee,\;\;y\mapsto t_y^*L\otimes L^{-1}.
\]

Let $\cP_A$ denote the normalized Poincaré bundle on $A\times A^\vee$, and  let $\Phi_A:\Db(A)\to \Db(A^\vee)$ be the Fourier-Mukai equivalence with kernel $\cP_A$. Then $\Phi_A(L)$ is a line bundle on $A^\vee$, and $\widehat{L}:=\Phi_A(L)^{-1}$ defines a principal polarization $\widehat{\theta}$ on $A^\vee$ such that $\varphi_{\widehat{\theta}}$ is the inverse of $\varphi_{\theta}$. 

The following well-known lemma justifies the notation $P_y$ for the line bundle $\varphi_\theta(y)$ on $A$:

\begin{lem}
    For any $y\in A$, $\varphi_\theta^*\left(\Phi_A(k(y))\right)=t_y^*L\otimes L^{-1}(=\varphi_\theta(y))$ as line bundles on $A$.
\end{lem}

If $A=JC$ is the Jacobian of a smooth projective curve $C$ of genus $g$, the above admits the following incarnation. We fix a point $p_0\in C$ and the corresponding Abel--Jacobi embedding
\[
    a:C\hookrightarrow JC,\;\;p\mapsto \cO_C(p-p_0),
\]
and we consider $L$ to be the line bundle represented by the divisor $\Theta:=C+\overset{g-1}{\dots}+C$. Then Jacobi inversion theorem asserts that $\varphi_{\widehat{\theta}}=\iota\circ a^*$, where $a^*:\Pic^0(JC)\to \Pic^0(C)=JC$ is the restriction map, and $\iota:JC\to JC$ denotes the natural inversion map. Given a point $\alpha\in JC$, we will denote by $\Theta_\alpha:=\Theta+\alpha$ the translated theta divisor.

The main character will be the rank-$g$ \emph{Picard bundle} $F:=\iota^*\varphi_\theta^*\left(R^1\Phi_{JC}(a_*\cO_C(-p_0))\right)$, whose fiber over a point $\alpha\in JC$ is canonically identified with
\[
H^1\left(JC,a_*\cO_C(-p_0)\otimes P_\alpha^{-1}\right)=H^1\left(C,\cO_C(-p_0)\otimes\alpha\right).
\]
Note that $F$ equals the vector bundle $F_{-1}$ defined in \cite[Section 2]{schw}; in particular, one has $\det(F)=\cO_{JC}(\Theta)$ as proven in \cite{mattuck,schw}. Moreover, for any $\alpha\in JC$
\begin{equation}\label{cohomologypicard}
h^i(JC,F\otimes P_\alpha)=\left\{
    \begin{array}{c l}
     \binom{g-1}{i} & \text{if $0\leq i\leq g-1$ and $\alpha\in C$}\\
     0 &  \text{if $i=g$ or $\alpha \notin C$.}\\
    \end{array}
    \right.
\end{equation}
(see for instance \cite[Proposition 4.4]{mukai-duality}\footnote{$\iota^*F$ equals the vector bundle  denoted by $F_{-1}$ in \cite{mukai-duality}. It is claimed in \emph{loc. cit} that the latter also coincides with the vector bundle $F_{-1}$ of \cite{schw}, but this holds only up to the action of $\iota$, due to Jacobi inversion.}).
It follows that the Picard bundle $F$ is a GV-sheaf in the sense of \cite{hacon,parpopaGV}. This implies that $F(\Theta)$ is $IT_0$ (the Index Theorem holds with index 0, in Mukai's language), in particular
\[
h^0(F(\Theta))=\chi(F(\Theta))=2g.
\]

\begin{rem} We could also consider the Picard bundle  $F_M:=\iota^*\varphi_\theta^*\left(R^1\Phi_{JC}(a_*M)\right)$ for any $M\in\Pic^{-1}(C)$; this is just the translate $t_{M(p_0)}^*F$. The analysis of \autoref{sec:positivity} remains valid (with appropriate changes) for the bundles $F_M(\Theta)$ and $\pi^*F_M(x)$, and the resulting  maps $JC\dasharrow \bP^g$ differ from ours just by composition with a translation on $JC$. 
\end{rem}

\section{Positivity of (twisted) Picard bundles}\label{sec:positivity}

In this section we will also consider the $g$-fold symmetric product $C^{(g)}$, together with the birational morphism
\[
\pi:C^{(g)}\longrightarrow JC,\qquad p_1+...+p_g\longmapsto \cO_C(p_1+...+p_g-gp_0).
\]
In $C^{(g)}$ we have the divisors $\pi^*\Theta$ and $x_p:=p+C^{(g-1)}\subset C^{(g)}$ for any $p\in C$, the latter ones being ample (see \cite[Proof of VII.2.2]{arbarello2013geometry}); for simplicity, we set $x:=x_{p_0}$.
Furthermore, the canonical divisor of $C^{(g)}$ coincides with the exceptional divisor of $\pi$, namely:
\[
K_{C^{(g)}}:=\big\{p_1+...+p_g\in C^{(g)}:\;h^0(C,\cO_C(p_1+...+p_g))\geq2\big\}
\]

The following lemma is well known; we add some details as we could not find a reference.

\begin{lem}\label{can_Cg}
    For every $\alpha=\cO_C(p-p_0)\in C\subset JC$, there is an equality $K_{C^{(g)}}+x_p=\pi^*(\Theta_\alpha)$ as divisors.
\end{lem}
\begin{proof}
It is immediate to check that $\pi^*(\Theta_\alpha)$ equals the divisor
\[
\big\{p_1+...+p_g\in C^{(g)}:\;h^0(C,\cO_C(p_1+...+p_g-p))\geq1\big\},
\]
which clearly contains both $K_{C^{(g)}}$ and $x_p$; therefore, we can write $\pi^*(\Theta_\alpha)=K_{C^{(g)}}+x_p+E$ for some effective divisor $E$. But since $\cO_{C^{(g)}}(\pi^*\Theta)\cong \omega_{C^{(g)}}\otimes \cO_{C^{(g)}}(x)$ (see \cite[Lemma 3.3]{polishchuk}), it turns out that $\pi^*(\Theta_\alpha)$ is numerically equivalent to $K_{C^{(g)}}+x_p$, which implies $E=0$.
\end{proof}

With a slight abuse of notation, we will say that $K_{C^{(g)}}+x_\alpha=\pi^*(\Theta_\alpha)$ for every $\alpha \in C$.

\vspace{2mm}

Consider the vector bundle $\widetilde{F}:=(\pi^*F)(x)$ on $C^{(g)}$. Note that $\det(\widetilde{F})=\cO_{C^{(g)}}(\pi^*\Theta+gx)$. The first goal of this section is to study the image of the determinant maps
\begin{align*}
\mathrm{det}_1:\bigwedge^g H^0\!\big(JC,F(\Theta)\big) \longrightarrow H^0\!\big(JC,\cO_{JC}((g+1)\Theta)\big), \\
\mathrm{det}_2:\bigwedge^g H^0\!\big(C^{(g)},\widetilde{F}\big) \longrightarrow H^0\!\big(C^{(g)},\cO_{C^{(g)}}(\pi^*\Theta+g x)\big).
\end{align*}
In particular, we will show that $\mathrm{Im}(\det_2)$ defines a base point free linear system, from which we will deduce that $\widetilde{F}$ is globally generated.

\medskip

To perform this study, observe that in virtue of \eqref{cohomologypicard}, for every $\alpha\in C$ we can consider a section $s_\alpha\in H^0(JC,F(\Theta))$, unique up to scalar, corresponding to the canonical inclusion $H^0(JC,F\otimes P_\alpha)\subset H^0(JC,F(\Theta))$ defined by $\Theta_\alpha$. In particular
\[
\Theta_\alpha\subset Z(s_\alpha),
\]
so it follows from \autoref{can_Cg} that $\pi^*s_\alpha\in H^0\left(C^{(g)},\pi^*(F(\Theta))\right)$ vanishes along the canonical divisor $K_{C^{(g)}}$ of $C^{(g)}$. Since $\widetilde{F}=\pi^*(F(\Theta))(-K_{C^{(g)}})$, $\pi^*s_\alpha$ induces a section $\widetilde{s_\alpha}\in H^0(C^{(g)},\widetilde{F})$ such that $x_\alpha\subset Z(\widetilde{s_\alpha})$.

Building on the generation properties studied in \cite{pareschi-generation}, we can now establish the following:

\begin{prop}\label{prop:decomposable}
The following hold.

\begin{enumerate}
\item\label{prop:decomp1} The linear system $|V|:=\!\bP\;\mathrm{Im}\big(\mathrm{det}_1\big) \subset |(g+1)\Theta|$ on $JC$
contains the divisor
\[
\Theta_{\alpha_1}+\Theta_{\alpha_2}+\cdots+\Theta_{\alpha_g}+\Theta_{\iota(\alpha_1+...+\alpha_g)}
\]
for every $\alpha_1,\dots,\alpha_g\in C$.

\vspace{2mm}

\item\label{prop:decomp2} The linear system $|W|:=\!\bP\;\mathrm{Im}\big(\mathrm{det}_2\big) \subset |\pi^*\Theta+g x|$ on $C^{(g)}$
contains the divisor
\[
x_{\alpha_1}+ x_{\alpha_2}+\cdots+ x_{\alpha_g}+ \pi^*\!(\Theta_{\iota(\alpha_1+\cdots+\alpha_g)})
\]
for every $\alpha_1,\dots,\alpha_g\in C$.
\end{enumerate}
\end{prop}

\begin{proof}
By \cite[Theorem D]{pareschi-generation}\footnote{With the notations in \cite{pareschi-generation}, our vector bundle $F$ equals the \emph{naive FMP transform} $\varphi_\theta^*\mathcal{T}(\omega_C(p_0))$.} the Picard bundle $F$ is \emph{generated by $\{C\}$}, namely the evaluation map
\[
\bigoplus_{\alpha\in U}H^0(JC,F\otimes P_\alpha)\otimes P_\alpha^{-1} \longrightarrow F
\]
is surjective for every open subset $U\subset JC$ intersecting $C$. Since $H^0(JC,F\otimes P_\alpha)=0$ for $\alpha\notin C$ by \eqref{cohomologypicard}, twisting by $\Theta$ results in a surjection
\begin{equation}\label{ev-surj}
    \bigoplus_{\alpha\in C}H^0(JC,F\otimes P_\alpha)\otimes \cO_{JC}(\Theta_\alpha) \longrightarrow F(\Theta).
\end{equation}
Recall that by construction $s_\alpha$ spans the image of
\[
H^0(JC,F\otimes P_\alpha)\otimes H^0\big(\cO_{JC}(\Theta_\alpha)\big)\longrightarrow H^0\!\big(F(\Theta)\big).
\]
Therefore, by surjectivity of \eqref{ev-surj}, for general points $\alpha_1,\dots,\alpha_g\in C$ the sections $s_{\alpha_1},\dots,s_{\alpha_g}\in H^0(JC,F(\Theta))$ are linearly independent over an open subset of $JC$; in other words, their determinant
$\det_1(s_{\alpha_1}\wedge\cdots\wedge s_{\alpha_g})\ \in\ H^0\!\big(\cO_{JC}((g+1)\Theta)\big)$
is nonzero. 

Since $Z(s_{\alpha_i})$ contains $\Theta_{\alpha_i}$, the inclusion
\[\Theta_{\alpha_1}+\cdots+\Theta_{\alpha_g}\subset Z\big(\mathrm{det} _1(s_{\alpha_1}\wedge\cdots\wedge s_{\alpha_g})\big)
\]
holds.
Moreover, the residual line bundle
\[
\cO_{JC}((g+1)\Theta)\otimes \cO_{JC}(-\Theta_{\alpha_1})\otimes\cdots\otimes \cO_{JC}(-\Theta_{\alpha_g})
\cong \cO_{JC}\!\big(\Theta_{\iota(\alpha_1+\cdots+\alpha_g)}\big)
\]
has a unique (up to scalar) nonzero global section, hence
\[
Z\big(\mathrm{det} _1(s_{\alpha_1}\wedge\cdots\wedge s_{\alpha_g})\big)
=\Theta_{\alpha_1}+\Theta_{\alpha_2}+\cdots+\Theta_{\alpha_g}+\Theta_{\iota(\alpha_1+...+\alpha_g)}
\]
for general $\alpha_1,\dots,\alpha_g\in C$. Since $|V|=\bP\;\mathrm{Im}(\det_1)$ is Zariski closed in $|(g+1)\Theta|$, it follows that the same divisor belongs to $|V|$ for all $\alpha_1,\dots,\alpha_g\in C$, which proves \eqref{prop:decomp1}.

Statement \eqref{prop:decomp2} is analogous. Indeed, one immediately checks that if $\det_1(s_{\alpha_1}\wedge\cdots\wedge s_{\alpha_g})\neq0$, then $\det_2(\widetilde{s_{\alpha_1}}\wedge\cdots\wedge \widetilde{s_{\alpha_g}})\neq0$
and moreover
\[
Z\!\big(\mathrm{det}_2(\widetilde{s_{\alpha_1}}\wedge\cdots\wedge \widetilde{s_{\alpha_g}})\big)
=
\pi^*Z\!\big(\mathrm{det}_1(s_{\alpha_1}\wedge\cdots\wedge s_{\alpha_g})\big)-g\cdot K_{C^{(g)}}.
\]
Using that $\pi^*(\Theta_{\alpha_i})=K_{C^{(g)}}+ x_{\alpha_i}$ (by \autoref{can_Cg}), it follows that 
\[
Z\!\big(\mathrm{det}_2(\widetilde{s_{\alpha_1}}\wedge\cdots\wedge \widetilde{s_{\alpha_g}})\big)
=x_{\alpha_1}+ x_{\alpha_2}+\cdots+ x_{\alpha_g}+ \pi^*\!(\Theta_{\iota(\alpha_1+\cdots+\alpha_g)})
\]
for general $\alpha_1,\dots,\alpha_g\in C$, and hence again the same divisor belongs to $|W|$ for all choices of $\alpha_1,\dots,\alpha_g\in C$, which finishes the proof.
\end{proof}

The main consequence of \autoref{prop:decomposable} is the global generation of $\widetilde{F}$:

\begin{prop}\label{prop:bp-free}
The linear system
$|W| \subset |\pi^*\Theta+g x|$
is base point free. In particular, $\widetilde{F}$ is globally generated.
\end{prop}

\begin{proof}
It is routine to verify that the family of divisors exhibited in \autoref{prop:decomposable}.\eqref{prop:decomp2} is base point free on $C^{(g)}$.
\end{proof}

Now we turn our attention to the study of the morphism
\[
j:C\longrightarrow \bP H^0(F(\Theta))=\bP^{2g-1},\;\;\;\;\alpha\mapsto[s_\alpha]=H^0(F\otimes P_\alpha)\otimes H^0(\cO_{JC}(\Theta_\alpha)).
\]

\vspace{2mm}

\begin{prop}\label{prop:morphismj}
The morphism $j$ is the closed immersion defined by the complete linear system $|\omega_C((g+1)p_0)|$ on $C$. 
\end{prop}
\begin{proof}
The morphism $j$ is set-theoretically injective, since $s_\alpha$ is the unique global section of $F(\Theta)$ (up to scalar) vanishing along $\Theta_\alpha$. To show this, it suffices to prove that the zero locus of the global section of $F\otimes P_\alpha$ has codimension $\geq2$. And indeed, for every effective divisor $D\subset JC$ we have
\[
\theta^{g-1}\cdot c_1\left((F\otimes P_\alpha)(-D)\right)=\theta^{g-1}\cdot (\theta-gD)=g!\left(1-[C]\cdot D\right)\leq0,
\]
whereas the vector bundle $(F\otimes P_\alpha)(-D)$ is slope stable with respect to the polarization $\theta$ by \cite{kempf}; this implies $\Hom\left(\cO_{JC},(F\otimes P_\alpha)(-D)\right)=0$.

For the equality $j^*\cO(1)=\omega_C((g+1)p_0)$ observe that, for every $\alpha\in JC$, cohomology and base change provides canonical isomorphisms of:
\begin{itemize}
    \item $H^0(\cO_{JC}(\Theta_\alpha))$ with the fiber at $\alpha$ of $\iota^*\varphi_\theta^*\Phi_{JC}(\cO_{JC}(\Theta))=\iota^*\cO_{JC}(-\Theta)$. 

    \item The Serre dual of $H^0(F\otimes P_\alpha)$ with the fiber at $\alpha$ of the sheaf $\varphi_\theta^*\iota^*R^g\Phi_{JC}(F^\vee)$.
\end{itemize}

According to \cite[3.8 and Theorem 3.13]{mukai-duality} and Grothendieck--Verdier duality, we have
\[
\varphi_\theta^*\iota^*R^g\Phi_{JC}(F^\vee)= \sHom(\varphi_\theta^*\Phi_{JC}(F),\cO_{JC})=\sExt^{g-1}(a_* \cO_C(-p_0),\cO_{JC})=a_*\left(\omega_C(p_0)\right).
\]
Therefore, the pullback $j^*\cO(-1)$ of the universal line bundle on $\bP H^0(F(\Theta))$ equals
\[
\iota^*\cO_{JC}(-\Theta)|_C\otimes \omega_C^{-1}(-p_0)=\omega_C^{-1}(-(g+1)p_0)
\]
(here we use $\iota^*\cO_{JC}(\Theta)|_C=\cO_C(g\cdot p_0)$), which shows that $j^*\cO(1)=\omega_C((g+1)p_0)$.

Since $\omega_C((g+1)p_0)$ is very ample and $h^0(C,\omega_C((g+1)p_0))=2g=h^0(JC,F(\Theta))$, to conclude the proof it suffices to check that the image $j(C)$ is non-degenerate in $\bP H^0(F(\Theta))$. Or equivalently, by the natural identification $H^0(JC,F(\Theta))=H^0\left(C^{(g)},\pi^*(F(\Theta))\right)$, it suffices to check that $j(C)\subset \bP H^0(C^{(g)},\widetilde{F}) \subset \bP H^0\left(C^{(g)},\pi^*(F(\Theta))\right)$ spans a linear subspace of dimension $\geq 2g-1$.

Let us denote by $\bP(M)\subset \bP H^0(C^{(g)},\widetilde F)$ the linear span of $j(C)$. It follows from the proofs of \autoref{prop:decomposable}.\eqref{prop:decomp2} and \autoref{prop:bp-free} that the evaluation map
\[
\mathrm{ev}_M:M\otimes \mathcal O_{C^{(g)}}\to \widetilde{F}
\]
is surjective, in particular $\dim \bP(M)=g-1+k$ for some $k\geq0$.
On the other hand, picking points $[\widetilde{s_{\alpha_1}}],\dots,[\widetilde{s_{\alpha_{g+k}}}]\in j(C)$ spanning $\mathbb P (M)$, the map $\mathrm{ev}_M$ cannot be surjective along the intersection
\[
Z(\widetilde{s_{\alpha_1}})\cap \dots \cap Z(\widetilde{s_{\alpha_{k+1}}}),
\]
since at any point of this intersection the image of $\mathrm{ev}$ is generated by the $g-1$ sections $\widetilde{s_{\alpha_{k+2}}},\dots,\widetilde{s_{\alpha_{g+k}}}\in M$. It follows that the intersection $Z(\widetilde{s_{\alpha_1}})\cap \dots \cap Z(\widetilde{s_{\alpha_{k+1}}})$ in $C^{(g)}$ is empty. But since $Z(\widetilde{s_{\alpha_i}})$ contains the ample divisor $x_{\alpha_i}$ for every $i$, such emptiness necessarily implies that $k\geq g$, which concludes the proof.
\end{proof}

In particular, the morphism $j$ gives us a better understanding of the global sections of $\widetilde{F}$:

\begin{cor}\label{cor-linearspan}
    There is a canonical identification $H^0(JC,F(\Theta))=H^0(C^{(g)},\widetilde F)$. Moreover, a global section $s\in H^0(\widetilde F)$ vanishes at a point $\zeta\in C^{(g)}$ (identified with a length-g subscheme of $C$) if and only if $[s]$ lies in the linear span $\langle j(\zeta)\rangle\subset \mathbb PH^0(\widetilde F)$.
    \label{vanishingloci}
\end{cor}
\begin{proof}
    The inclusion $H^0(C^{(g)},\widetilde F)\hookrightarrow H^0\left(C^{(g)},\pi^*(F(\Theta))\right)=H^0(JC,F(\Theta))$ is an equality as shown in \autoref{prop:morphismj} (both sides are spanned by the curve $C$).
    For the second assertion, we start noting the inclusion
    \begin{equation}\label{eq:inclusion}
    \langle j(\zeta)\rangle \subset \bP H^0(C^{(g)},\widetilde{F}\otimes \mathcal I_{\zeta}).
    \end{equation}
   Indeed, \eqref{eq:inclusion} holds true when $\zeta$ is supported on $g$ distinct points $p_1,...,p_g\in C$ (since the global sections $\widetilde{s_{p_i}}\in H^0(C^{(g)},\widetilde{F})$ vanish along $x_{p_i}$, which contains $\zeta$) and so it holds true for every $\zeta\in C^{(g)}$. Thus the statement is equivalent to \eqref{eq:inclusion} being an equality.
    
    On the one hand, the line bundle $\omega_C((g+1) p_0)$ is $(g-1)$-very ample (as it has degree $3g-1$), so in virtue of \autoref{prop:morphismj} the span $\langle j(\zeta)\rangle\subset \bP H^0(C^{(g)},\widetilde{F})$ is $(g-1)$-dimensional. On the other hand, $\bP H^0(C^{(g)},\widetilde{F}\otimes \mathcal I_{\zeta})$ also has dimension $g-1$ (since $\widetilde{F}$ is globally generated with $h^0(\widetilde{F})=2g$), and so \eqref{eq:inclusion} is indeed an equality.
\end{proof}

We can also deduce that there are no unexpected loci where global sections of $\widetilde F$ become linearly dependent:

\begin{prop}\label{prop:maximalrank}
For any $k\ge g$ and $W\in \mathrm{Gr}(k,H^0(\widetilde F))$, the locus where the evaluation  map
\[
\mathrm{ev}_W:W\otimes \mathcal O_{C^{(g)}}\longrightarrow \widetilde F
\]
does not have maximal rank is of codimension at least $1+k-g$.
\end{prop}

\begin{proof}
By \autoref{prop:morphismj} and \autoref{cor-linearspan}, we can pick points $\alpha_1,\dots,\alpha_{2g-k}\in C$ such that
\[
W\oplus \langle \widetilde s_{\alpha_1},\dots,\widetilde s_{\alpha_{2g-k}}\rangle = H^0(\widetilde F).
\]
Denoting by $B_W\subset C^{(g)}$ the locus where $\mathrm{ev}_W$ is not of maximal rank, observe that the full evaluation map
\[
\mathrm{ev}:H^0(\widetilde F)\otimes \mathcal O_{C^{(g)}}\longrightarrow \widetilde F
\]
fails to be surjective along the intersection
\[
B:=B_W\cap Z(\widetilde s_{\alpha_1})\cap \dots \cap Z(\widetilde s_{\alpha_{2g-k}}).
\]
Since $\mathrm{ev}$ is surjective by \autoref{prop:bp-free}, it follows that $B$ must be empty.
On the other hand, each zero locus $Z(\widetilde s_{\alpha_i})$ contains the ample divisor $x_{\alpha_i}$; therefore, $B$ is non-empty as soon as $\mathrm{codim}(B_W)\leq k-g$. This implies $\mathrm{codim}(B_W)\ge 1+k-g$, as desired.
\end{proof}

\begin{cor}
The determinant map
\[
\mathrm{det}_2:\bigwedge^g H^0(\widetilde F)\longrightarrow H^0\!\bigl(\mathcal O_{C^{(g)}}(\pi^*\Theta+g x)\bigr)
\]
is non-zero on decomposable tensors. In particular, it is of rank at least $g^2+1$. 
\end{cor}
\begin{proof}
Applying \autoref{prop:maximalrank} for $k=g$, we obtain that the determinant map $\det_2$ vanishes on no decomposable tensor. Therefore, the Grassmannian $\mathrm{Gr}(g,H^0(\widetilde F))$ does not intersect the linear subspace $\mathbb P (\ker\det_2)\subset \mathbb P\bigwedge^g H^0(\widetilde F)$. Since $\dim \mathrm{Gr}(g,H^0(\widetilde F))=g^2$, it follows that $\mathrm{codim}(\ker\det_2)\ge g^2+1$, or equivalently, that $\rk(\det_2)\ge g^2+1$.
\end{proof}

\section{Application to the degree of irrationality of Jacobians}

With the study performed in the previous section, we are now in a position to prove \autoref{thm:intro-jacobian-bound}:

\begin{proof}[Proof of \autoref{thm:intro-jacobian-bound}]
In view of \autoref{deg-gg} and \autoref{prop:bp-free}, to establish that $\irr(JC)=\irr(C^{(g)})\leq 2^g$ it suffices to show that $c_g(\widetilde{F})=2^g$. 

To perform this computation, observe that $c_i(F)=[W_{g-i}(C)]$ in the Chow group of $JC$ by \cite{mattuck,schw}, where $W_{g-i}(C)=C+\overset{g-i}{...}+C$. In particular, 
\[
c_i(\pi^*F)=\frac{(\pi^*\theta)^i}{i!}\in H^{2i}(C^{(g)},\bZ).
\]
On the other hand, denoting also by $x$ the cohomology class $c_1(\cO_{C^{(g)}}(x))\in H^2(C^{(g)},\bZ)$, it is well known that
\[
        x^k(\pi^*\theta)^{g-k}=\frac{g!}{k!}
\]
(see \cite[Lemma 1]{kouvi}). Applying the splitting principle (as in \cite[Remark 3.2.3]{fulton}) yields
\[
c_g(\widetilde{F})=c_g(\pi^*F\otimes \mathcal O_{C^{(g)}}(x))=\sum_{k=0}^g x^k\frac{ (\pi^*\theta)^{g-k}}{(g-k)!}=\sum_{k=0}^g\binom{g}{k}=2^g,
\]
which completes the proof.
\end{proof}

In other words, via the correspondence described in \autoref{compdegree}, we have shown that a general element $V^\vee\in \mathrm{Gr}(g+1,H^0 (\widetilde{F}))$ induces a rational map $\varphi_V:C^{(g)}\dashrightarrow \mathbb P^g$ which is dominant of degree $2^g$. In view of \autoref{cor-linearspan}, we can regard $\varphi_V$ more geometrically as follows. We have the embedded curve
\[
j:C\hookrightarrow |\omega_C((g+1)p_0)|^\vee=\bP^{2g-1}
\]
and we have fixed a general $g$-dimensional linear space $\bP(V^\vee)\subset \bP^{2g-1}$.
Then for a general $p_1+...+p_g\in C^{(g)}$, the $(g-1)$-dimensional span $\langle j(p_1),...,j(p_g)\rangle$ intersects $\bP(V^\vee)$ at exactly one point; we define $\varphi_V(p_1+...+p_g)\in\bP(V^\vee)$ to be this intersection point.

It is natural to ask whether for every choice of the subspace $V^\vee$, the map $\varphi_V$ remains dominant. More interestingly, one can also ask whether $\deg(\varphi_V)$ may drop for special choices of $V^\vee$ (which would result in a better upper bound for $\irr(JC)$). The following lemma answers these questions in the case $g=2$:

\begin{lem}
For every $V^\vee\in \mathrm{Gr}\left(g+1,H^0(\widetilde{F})\right)$, $(\widetilde{F},V^\vee)$ defines a good pair. Furthermore, if $g=2$ then $(\widetilde{F},V^\vee)$ induces a rational dominant map $\varphi_V$ of degree $4$.
\end{lem}
\begin{proof}
For the first claim, simply observe that any  $V^\vee\in \mathrm{Gr}\left(g+1,H^0(\widetilde{F})\right)$ generates $\widetilde{F}$ in codimension 1, as a direct application of \autoref{prop:maximalrank} for $k=g+1$.

Now assume $g=2$ and consider any $V^\vee\in \mathrm{Gr}(3,H^0(\widetilde{F}))$. In view of \autoref{vanishingloci}, a section of $\mathbb PV^\vee \setminus C$  has vanishing locus of dimension $0$. Hence, the fiber computation in \autoref{compdegree} shows that $\varphi_V$ is dominant, and $\deg(\varphi_V)<c_2(\widetilde{F})=4$ if and only if all sections of $V^\vee$ vanish at a common point; namely, if and only if $V^\vee \subset H^0(\widetilde{F}\otimes \mathcal I_{\zeta})$
for some $\zeta\in C^{(2)}$.

But note that the latter condition cannot happen: since $h^0(\widetilde{F})=4$ and $\widetilde{F}$ is globally generated (recall \autoref{prop:bp-free}), we have $h^0(\widetilde{F}\otimes \mathcal I_{\zeta})=2$ for every $\zeta\in C^{(2)}$. It follows that $\varphi_V$ has degree $4$ for every $V^\vee\in \mathrm{Gr}(3,H^0(\widetilde{F}))$.
\end{proof}

\section{The case of Prym varieties}

In this section we apply a similar formalism to the study of Prym varieties, which will result in the upper bound of \autoref{thm:intro-prym-bound}.

\subsection{Basics on Prym varieties}
 Let us first recall some standard facts on Prym varieties; the reader may consult the seminal work \cite{mumford} or \cite[Chapter 12]{birklange} for details. Let $f:\tC\to C$ be an irreducible étale double cover of a smooth projective curve $C$ of genus $g$. The curve $\tC$ has genus $2g-1$, and comes with an involution $\sigma:\tC\to \tC$ having no fixed point. The kernel of the norm map
\[
\Nm_f:J\tC\to JC
\]
has two connected components; the one containing the origin of $J\tC$ is called the \emph{Prym variety} $\Pr:=\Pr(f)$. It has dimension $g-1$ and is naturally equipped with a principal polarization $\xi\in \NS(\Pr)$ such that $2\xi$ equals the restriction of the canonical principal polarization $\widetilde{\theta}$ on $J\tC$. 

The Prym variety can also be realized as the image of the endomorphism $\Id-\sigma$ of $J\tC$; denoting by $j$ the inclusion $\Pr\hookrightarrow J\tC$, one has $j^\vee\circ \varphi_{\widetilde{\theta}}=\varphi_\xi\circ (\Id-\sigma)$. In other words, for every $\alpha\in J\tC$: 
\begin{equation}\label{restriction-prym}
P_{\alpha-\sigma(\alpha)}={{P_\alpha}_|}_{\Pr}\in {\Pr}^\vee.
\end{equation}

If the curve $C$ is hyperelliptic, then $\Pr$ is a product of Jacobians (see \cite[Section 7]{mumford}) and the upper bound $\irr(\Pr)\leq 2^{2g-3}$ follows immediately from \autoref{thm:intro-jacobian-bound}. For this reason, in the sequel we will assume that $C$ is not hyperelliptic. 

For a chosen point $p_0\in\tC$, we consider the \emph{Abel--Prym map}
\[
b: \tC\hookrightarrow \Pr, \;\;\;p\mapsto \cO_{\tC}\left(p-\sigma(p)-p_0+\sigma(p_0)\right)
\]
which embeds $\tC$ in $\Pr$ as a curve of class $2\frac{\xi^{g-2}}{(g-2)!}\in H^{2g-4}(\Pr,\bZ)$. Note that $b$ is the composition of the Abel--Jacobi embedding $a:\tC\hookrightarrow J\tC$ (with base point $p_0$) and $\Id-\sigma:J\tC\to \Pr$.

For arbitrary $d$, the norm map extends to a map $\Pic^d(\tC)\to \Pic^d(C)$, which we also denote by $\Nm_f$. In particular, the fiber $\Nm_f^{-1}(\omega_C)$ splits as a disjoint union of two natural torsors of $\Pr$:
\begin{align*}
{\Pr}^+:=\{L\in\Pic^{2g-2}(\tC):\; \Nm_f(L)=\omega_C,\;\text{$h^0(\tC,L)$ is even}\},\\
{\Pr}^-:=\{L\in\Pic^{2g-2}(\tC):\; \Nm_f(L)=\omega_C,\;\text{$h^0(\tC,L)$ is odd}\}.
\end{align*}
These torsors provide realizations of two important loci: a canonical theta divisor $\Xi\subset {\Pr}^+$, and the \emph{Prym--Brill--Noether locus} $V^2(f)\subset {\Pr}^-$ (see \cite{welters}). The latter is well known to have codimension 3, see for example \cite[Theorem 2.2]{clv}.

We choose, once and for all, an element $M\in \Pr^-$. In this way, via tensor product with $M$ (resp.~with $M(\sigma(p_0)-p_0)$) we will freely identify $\Pr$ with $\Pr^-$ (resp.~with $\Pr^+$), allowing us to regard both the theta divisor $\Xi$ and $V^2(f)$ as subvarieties of $\Pr$.

Consider the variety $X^-$ defined by the cartesian diagram
\[
\begin{tikzcd}
X^- \arrow{d}\arrow[hookrightarrow]{r} & \tC^{(2g-2)} \arrow{d} \\
{\Pr}^- \arrow[hookrightarrow]{r}                          & \Pic^{2g-2}(\tC), \
\end{tikzcd}
\]
we will denote by $\pi$ the corresponding birational morphism $X^-\to \Pr^-\cong \Pr$.
It is an isomorphism over the complement of $V^2(f)\subset \Pr$, and will play the role of the Abel--Jacobi map $C^{(g)}\to JC$ in the previous sections. Indeed, \autoref{thm:intro-prym-bound} will be deduced from a globally generated vector bundle of rank $g-1$ on (a desingularization of) $X^-$, obtained as the twist of the pullback of a Picard type bundle on $\Pr$.

For any line bundle $L\in \Pic^{g-1}(\tC)$ satisfying the condition $h^0(C,\Nm_f(L))=0$, let us consider the \emph{Prym--Picard sheaf} $G:=\iota^*\varphi_{\xi}^*\left(R^1\Phi_{\Pr}(b_*L)\right)$ on $\Pr$.

\begin{prop}\label{prympicard}
    The Prym--Picard sheaf $G$ is locally free of rank $g-1$, with Chern classes
    \[
    c_i(G)=2^i\frac{\xi^i}{i!}\in H^{2i}(\Pr,\bZ)
    \]
    for every $1\leq i \leq g-1$. Furthermore, for every $\beta\in \Pr$ the following hold:
    \[
    h^i(\Pr,G\otimes P_\beta)=\left\{
    \begin{array}{c l}
     \binom{g-2}{i} & \text{if $0\leq i\leq g-2$ and $\beta\in \tC$}\\
     0 &  \text{if $i=g-1$ or $\beta \notin \tC$.}\\
    \end{array}
    \right.
    \]
\end{prop}
\begin{proof}
Cohomology and base change, combined with Jacobi inversion for Pryms (see e.g.~\cite[Lemma 3.2]{naranjo}), yields a canonical identification of the fiber of $G$ at any point $\beta\in \Pr$ with
\[
H^1\left(\Pr,b_*L\otimes P_\beta^{-1}\right)=H^1\left(\tC,L\otimes\beta\right).
\]
Since the norm of an effective line bundle must be effective, $h^0(C,\Nm_f(L))=0$ implies that $h^0(\tC,L\otimes \beta)=0$ for every $\beta\in\Pr$.  Therefore, $G$ is locally free of rank $-\chi(\tC,L)=g-1$. 

Furthermore, $R^0\Phi_{\Pr}(b_*L)$ being reflexive (it is the pushforward of a line bundle on $\tC\times  \Pr^\vee$), it must be torsion-free and hence the equality $R^0\Phi_{\Pr}(b_*L)=0$ holds. Accordingly,
\[
G[-1]=\iota^*\varphi_\xi^*\left(\Phi_{\Pr}(b_*L)\right).
\]
Since  $\ch(b_* L)=(0,...,0,2\frac{\xi^{g-2}}{(g-2)!},1-g)$, using \cite[Proposition 1.17]{mukai2} one obtains the Chern character $\ch(G)=(g-1,2\xi,0,...,0)$,
from  which the Chern classes of $G$ can be easily recovered.

For the cohomological assertion, a similar argument to the one in \cite[Proposition 4.4.(1)]{mukai-duality} can be applied.
\end{proof}

\begin{rem}\label{rem:picard-restriction}
The bundle $G$ is the restriction to $\Pr$ of the following Picard sheaf on $J\tC$:
\[
K_L:=\iota^*\varphi_{\widetilde{\theta}}^*\left(R^1\Phi_{J\tC}(a_*L)\right)=t_{\gamma}^*F_{g-1},
\]
where  $\gamma=L(-(g-1)p_0)\in J\tC$ and $F_{g-1}$ is the Picard sheaf of rank $g-1$ on $J\tC$ defined in \cite[Section 2]{schw}. In particular, using the computation for $\det(F_{g-1})$  in \cite{mattuck,schw} it follows that
$\det(K_L)=\cO_{J\tC}(\Theta^L)$, where
\[
\Theta^L=\Bigl\{\alpha\in\Pic^0(\tC)=J\tC:\; h^0\left(\tC,L((g-1)p_0)\otimes\alpha\right)>0\Bigr\}\subset J\widetilde C.
\]
\end{rem}

\subsection{Proof of \autoref{thm:intro-prym-bound}} 
Given a point $p\in \widetilde C$, we consider
$x_p := p+\widetilde C^{(2g-3)} \subset \widetilde C^{(2g-2)}$. The following two lemmas are probably well known to the experts, but we add some details since we could not find a reference.

\begin{lem}
\label{remark}
The intersection $x_p \cap X^-$ is reduced for every point $p\in\tC$. 
\end{lem}
\begin{proof}
First we check the assertion for a general $p\in\tC$. Let $\zeta=p_1+\cdots+p_{2g-2}\in X^-$ be a smooth point, such that the points $p_i\in\tC$ are all distinct. The intersection of Zariski tangent spaces
$T_\zeta (x_{p_1})\cap\cdots\cap T_\zeta (x_{p_{2g-2}})$ inside $T_\zeta(\tC^{2g-2})$ is $0$, and so there exists an index $i$ such that $T_\zeta (X^-)\not\subset T_\zeta (x_{p_i})$. It follows that, for a general $p\in \tC$ and $\zeta\in x_p\cap X^-$, the intersection $x_p\cap X^-$ is transverse at $\zeta$, and so $x_p \cap X^-$ is reduced.

To derive the assertion for any $p$, note that the map $x_p\cap X^-\to \mathrm{Pr}^-$ is generically one-to-one onto its image, so $x_p \cap X^-$ is reduced if and only if its schematic image $\pi(x_p \cap X^-)$ is reduced. 

The image $\pi(x_p \cap X^-)$ in $\mathrm{Pr}$ can be identified set-theoretically with Beauville’s special subvariety \cite{beauville} defined by the complete linear system $|\omega_C(-f(p))|$ on $C$. The latter is reduced, since by the main result of \emph{loc.\ cit.}, its cohomology class is $\xi$ (which is non-divisible). Since ${x_p}_{|_{X^-}}$ is reduced for general $p$, the identification holds scheme theoretically (for every $p\in\tC$), and the lemma follows.
\end{proof}

By abuse of notation, from now on we write $x_p$ for ${{x_p}_|}_{X^-}$.
We will also write $x:=x_{p_0}$.

\begin{lem}\label{lemma52}
    For every $p\in\tC$ and $\beta=b(p)\in \widetilde C\subset \Pr$, there is an equality $E+x_p=\pi^*(\Xi_\beta)$, where $E$ is supported on the exceptional divisor of $\pi$.
\end{lem}
\begin{proof}

Under the previous identifications, one can readily check that, set-theoretically, $\pi^*(\Xi_\beta)$ coincides with the divisor
\[
\Bigl\{
p_1+\cdots+p_{2g-2}\in X^- \subset \widetilde C^{(2g-2)}
:\;
h^0\!\left(
\widetilde C,
\cO_{\widetilde C}(p_1+\cdots+p_{2g-2}+\sigma(p)-p)
\right)\ge 2
\Bigr\}.
\]
Since $\cO_{\widetilde C}(p_1+\cdots+p_{2g-2}+\sigma(p)-p)\in {\Pr}^+$ for every $p_1+\cdots+p_{2g-2}\in X^-$, it follows that $x_p\subset \pi^*(\Xi_\beta)$ by reducedness of $x_p$. And since  $\pi^*\Xi_\beta$ is reduced away from the exceptional divisor of $\pi$,  we deduce that $\pi^*(\Xi_\beta)-x_p$ is effective and does not contain $x_p.$

Now we consider a point
$q_1+\cdots+q_{2g-2}$ in $\pi^*(\Theta_\beta)\setminus x_p$. 
Since it does not lie in $x_p$, we must have $p\neq q_i$ for all $i$, whereas certainly $p\neq \sigma(p)$. It follows that
\[
h^0\!\left(
\widetilde C,
\cO_{\widetilde C}(q_1+\cdots+q_{2g-2}+\sigma(p)-p)
\right)
<
h^0\!\left(
\widetilde C,
\cO_{\widetilde C}(q_1+\cdots+q_{2g-2}+\sigma(p))
\right),
\]
where the left-hand side is $\geq2$. Therefore,
\[
h^0\!\left(
\widetilde C,
\cO_{\widetilde C}(q_1+\cdots+q_{2g-2})
\right)\ge 2
\]
and since $q_1+\cdots+q_{2g-2}\in X^-$, the parity condition along ${\Pr}^-$ implies that
\[
h^0\!\left(
\widetilde C,
\cO_{\widetilde C}(q_1+\cdots+q_{2g-2})
\right)\ge 3.
\]
Consequently $\pi(q_1+\cdots+q_{2g-2})\in V^2(f)$, that is, $q_1+\dots +q_{2g-2}$ belongs to the exceptional divisor of $\pi$.
\end{proof}

Now we consider the vector bundle $(\pi^*G)(x)$ on $X^-$ and the determinant map
\[
\det:\bigwedge^{g-1}H^0\!\bigl(X^-,(\pi^*G)(x)\bigr)\longrightarrow
H^0\!\left(X^-,\det\bigl((\pi^*G)(x)\bigr)\right),
\]
whose image is the key point to establish the following:

\begin{prop}
    The vector bundle $(\pi^*G)(x)$ on $X^-$ is globally generated.
    \label{globalgenerationprym}
\end{prop}

\begin{proof}
   The proof applies the same strategy of \autoref{prop:decomposable}\eqref{prop:decomp2}. First, using \autoref{prympicard} we see that for every $\beta\in\tC$ there exists a (unique) global section $s_\beta\in H^0(G(\Xi))$ such that $\Xi_\beta\subset Z(s_\beta)$. The Picard sheaf $G$ is again generated by $\{\tC\}$ in the sense of \cite{pareschi-generation}, and so for general points $\beta_1,...,\beta_{g-1}\in\tC$, one has
   \[
   Z\left(\det(s_{\beta_1}\wedge ...\wedge s_{\beta_{g-1}})\right)=
\Xi_{\beta_1}+\cdots+ \Xi_{\beta_{g-1}}+R(\beta_1,...,\beta_{g-1})
   \]
   for a suitable residual divisor $R(\beta_1,\dots,\beta_{g-1})$ of class $2\xi$.
   Then, in view of \autoref{lemma52}, considering  $(\pi^*G)(x)$, it follows that the linear system $|\mathrm{Im}(\mathrm{det})|$ on $X^-$ contains the divisors
   \[
x_{\beta_1}+ x_{\beta_2}+\cdots+ x_{\beta_{g-1}}+ \pi^*R(\beta_1+\cdots+\beta_{g-1}),
   \]
   for general $\beta_1,\dots,\beta_{g-1}\in \widetilde C$. Hence the base locus of $|\mathrm{Im}(\mathrm{det})|$ is contained in $\pi^*B$, where  
   \[
   B:=\bigcap R(\beta_1,\dots,\beta_{g-1})
   \]
(the intersection running over general points $\beta_1,...,\beta_{g-1}\in\tC$).  It follows that $(\pi^*G)(x)$ is globally generated away from $\pi^*B$, so to finish the proof it suffices to check that $B$ is empty.
   
To this end observe that, as pointed out in \autoref{rem:picard-restriction}, the vector bundle $G$ is the restriction to $\mathrm{Pr}$ of the Picard sheaf $K_L$ on $J\widetilde C$, whose determinant is the divisor $\Theta^L$.

For any $p\in \widetilde C$, we set $\alpha=a(p)$ and $\beta=b(p)$. Then the unique (up to scalar) map $P_\beta^{-1}\to G$ on $\mathrm{Pr}$ is the restriction of the map $P_\alpha^{-1}\to K_L$ on $J\tC$ (recall \eqref{restriction-prym}). If $\beta_1,...,\beta_{g-1}\in\tC$ are general points, it follows that the residual divisor $R(\beta_1,\dots,\beta_{g-1})$---which is the locus where $\bigoplus_{i=1}^{g-1}P_{\beta_i}^{-1}\to G$ drops rank---is the restriction to $\mathrm{Pr}$ of the corresponding  degeneracy locus for $\bigoplus_{i=1}^{g-1}P_{\alpha_i}^{-1}\to K_L$. The latter is, for determinant reasons, the translate $\Theta^L_{\iota(\alpha_1+\cdots+\alpha_{g-1})}$ of $\Theta^L$. Hence
\[
R(\beta_1,\dots,\beta_{g-1})
=
\Theta^L_{\iota(\alpha_1+\cdots+\alpha_{g-1})}\cap \mathrm{Pr},
\]
which implies
\[
B=\bigcap_{\alpha_1,\dots,\alpha_{g-1}\in \widetilde C}
\Bigl(\Theta^L_{\iota(\alpha_1+\cdots+\alpha_{g-1})}\cap \mathrm{Pr}\Bigr)
=
\left(\bigcap_{\alpha_1,\dots,\alpha_{g-1}\in \widetilde C}\Theta^L_{\iota(\alpha_1+\cdots+\alpha_{g-1})}\right)\cap \mathrm{Pr}.
\]
An elementary computation shows that any $\alpha\in\Pic^0(\tC)=J\tC$ in this intersection has trivial norm and satisfies that $L\otimes \alpha$ is effective. If such an $\alpha$ existed, then $\mathrm{Nm}_f(L)=\Nm_f(L\otimes \alpha)$ would be effective, contradicting our assumption on $L$; therefore, $B$ is empty and the proof is complete.
\end{proof}

We let $\nu:\widetilde{X^-}\to X^-$ be a desingularization of $X^-$, and denote $\psi:=\nu\circ \pi$. Consider the vector bundle $\widetilde G=\nu^*((\pi^*G)(x))$.

\begin{prop}
   The vector bundle $\widetilde G$ is globally generated and satisfies $c_{g-1}(\widetilde{G})=2^{2g-3}$.
    \label{chernclassprym}
\end{prop}
\begin{proof}
The global generation of $\widetilde G$ follows directly from the previous proposition. 

For the computation of $c_{g-1}(\widetilde{G})$, first observe that, in virtue of \autoref{prympicard}, 
\[
c_i(\psi^*G)=2^i\frac{\psi^*\xi^i}{i!}\in H^{2i}(\widetilde {X^-},\bZ).
\]
On the other hand, denoting also by $x$ the cohomology class $c_1(\mathcal{O}_{\widetilde{X^-}}(\nu^*x))\in H^2(\widetilde{X^-},\bZ)$, we claim that $\psi_*x^k=2^{k-1}\frac{\xi^{k}}{k!}$. Indeed, fix general points $p_1, \dots, p_k\in \widetilde C$. By a similar argument to that of the proof of \autoref{remark} the intersection $x_{p_1}\cap \dots \cap x_{p_k}\subset X^-$ is reduced, and maps generically 1:1 onto its image in $\mathrm{Pr}$. This image can in turn be identified with Beauville's special subvariety defined by the complete linear system $|\omega_C(-f(p_1)-...-f(p_k))|$ on $C$, and so it has cohomology class $2^{k-1}\frac{\xi^{k}}{k!}$ (see \cite{beauville}).

Therefore, a direct application of projection formula yields
\[
         x^k\cdot (\psi^*\xi)^{g-1-k}=\psi_*x^k\cdot \xi^{g-1-k}=2^{k-1}\frac{(g-1)!}{k!}.
\]

Applying the splitting principle (again, as in \cite[Remark 3.2.3]{fulton}) yields
\[
c_{g-1}(\widetilde G)=c_{g-1}({\psi}^*G\otimes \mathcal O_{\widetilde X^-}(\nu^*x))=\sum_{k=0}^{g-1}2^{g-1-k} x^k\frac{ (\psi^*\xi)^{g-1-k}}{(g-1-k)!}=2^{g-2}\sum_{k=0}^{g-1}\binom{g-1}{k}=2^{2g-3},
\]
which completes the proof.
\end{proof}

Observe that at this point, \autoref{thm:intro-prym-bound} follows directly by combining \autoref{chernclassprym} with \autoref{deg-gg}.

\bibliography{refer}
\bibliographystyle{alphaspecial}
\end{document}